\RequirePackage[l2tabu, orthodox]{nag}

\documentclass[12pt]{amsart}
\usepackage{fullpage,url,amssymb,enumerate,colonequals}
\usepackage[all]{xy} 
\usepackage{comment}
\usepackage{graphicx}

\usepackage[OT2,T1]{fontenc}

\usepackage{color}

\def\bC{\mathbb{C}}

\def\bH{\mathbb{H}}
\def\cH{\mathcal{H}}

\def\cM{\mathcal{M}}

\def\bP{\mathbb{P}}

\def\bQ{\mathbb{Q}}

\def\bZ{\mathbb{Z}}

\def\barM{\overline{\cM}}
\def\barH{\overline{\cH}}

\def\wt{\widetilde}

\DeclareMathOperator{\res}{res}

\DeclareMathOperator{\SL}{SL}
\DeclareMathOperator{\Spec}{Spec}

\newtheorem{thm}{Theorem}[section]
\newtheorem{lem}[thm]{Lemma}
\newtheorem{cor}[thm]{Corollary}
\newtheorem{prop}[thm]{Proposition}

\newtheorem{conjecture}{Conjecture}

\theoremstyle{definition}
\newtheorem{rem}[thm]{Remark}

\newtheorem{defn}[thm]{Definition}

\makeatletter
\g@addto@macro\bfseries{\boldmath} 
\makeatother

\usepackage{microtype}

\usepackage[
	backref,
	pdfauthor={Carl Lian}, 
	pdftitle={Non-tautological Hurwitz cycles},
]{hyperref}

\begin{document}

\title{Non-tautological Hurwitz cycles}

\author{Carl Lian}
\address{Institut f\"{u}r Mathematik, Humboldt-Universit\"{a}t zu Berlin, 12489 Berlin, Germany}
\email{liancarl@hu-berlin.de}
\urladdr{\url{https://sites.google.com/view/carllian}}

\date{\today}

\begin{abstract}
We show that various loci of stable curves of sufficiently large genus admitting degree $d$ covers of positive genus curves define non-tautological algebraic cycles on $\barM_{g,N}$, assuming the non-vanishing of the $d$-th Fourier coefficient of a certain modular form. Our results build on those of Graber-Pandharipande and van Zelm for degree 2 covers of elliptic curves; the main new ingredient is a method to intersect the cycles in question with boundary strata, as developed recently by Schmitt-van Zelm and the author.
\end{abstract}

\maketitle


\section{Introduction}\label{intro}

\subsection{Tautological classes on moduli spaces of curves}

The Chow $A^{*}(\barM_{g,n})$ and cohomology $H^{*}(\barM_{g.n})$ rings of moduli spaces of stable pointed curves are central objects of enumerative geometry. While both objects are extremely complicated and likely impossible to understand completely, Mumford \cite{mumford} initiated a study of certain \textit{tautological classes} on $\barM_{g,n}$ that appear in many natural geometric situations and are largely computable in practice.

By definition, the tautological rings $R^{*}(\barM_{g,n})\subset A^{*}(\barM_{g,n})$ form the smallest system of subrings containing the $\psi$ and $\kappa$ classes and closed under all pushforwards by forgetful morphisms $\pi:\barM_{g,n+1}\to\barM_{g,n}$ and boundary morphisms $\xi_{\Gamma}:\barM_{\Gamma}\to\barM_{g,n}$. Moreover, additive generators for the tautological ring and formulas for their intersections may be given combinatorially, see \cite[Appendix A]{gp}. A conjecturally complete set of relations is given by Pixton's relations, see \cite{ppz}.

Many cohomology classes on moduli spaces of curves arising in geometry turn out to be tautological. For example, using techniques of Gromov-Witten theory, Faber-Pandharipande \cite{fp} show that loci of curves admitting maps to $\bP^1$ with prescribed ramification profiles are tautological.

We review the theory of tautological classes in \S\ref{taut_review}.

\subsection{Non-tautological classes from Hurwitz cycles}

In contrast to the result of \cite{fp}, it was first shown by Graber-Pandharipande \cite{gp} that certain loci of curves admitting double covers of positive genus curves are non-tautological. For example:

\begin{thm}\cite[Theorem 2]{gp}
The locus of pointed curves $[X,x_1,\ldots,x_{20}]\in \barM_{2,20}$ such that there exists a 2-to-1 cover $f:X\to E$, where $E$ is a genus 1 curve, and $f(x_{2i-1})=f(x_{2i})$ for $i=1,\ldots,10$, is non-tautological.
\end{thm}

More recently, this result was extended by van Zelm:

\begin{thm}\cite[Theorem 1]{vanzelm}
Suppose that $g\ge2$ and $g+m\ge12$. Then, the locus of pointed curves on $\barM_{g,2m}$ admitting a double cover of an elliptic curve with $m$ pairs of conjugate points, is non-tautological.
\end{thm}

In particular, when $g\ge12$, one obtains non-tautological classes on $\barM_{g}$.

The method as follows: suppose first that $g+m=12$. Then, consider the boundary stratum $\xi:\barM_{1,11}\times\barM_{1,11}\to\barM_{g,2m}$ gluing together $g-1$ pairs of points on opposite components. By \cite[Proposition 1]{gp} (see also Proposition \ref{nontaut_criterion}), the pullback of any tautological class to a boundary stratum has K\"{u}nneth decomposition (in cohomology) into tautological classes. However, a combinatorial calculation shows that the pullback of the pointed bielliptic class is a non-zero multiple of the class of the diagonal $\barM_{1,11}\to\barM_{1,11}\times\barM_{1,11}$, which cannot have tautological K\"{u}nneth decomposition owing to the existence of odd cohomology on $\barM_{1,11}$, see \S\ref{m11_intro}.

When $g+m>12$, one can induct on $g$ and use the same criterion with different boundary strata to conclude, see \cite[Lemma 12]{vanzelm}.

\subsection{New results}

The goal of this paper to extend these results further to loci of curves (\textit{Hurwitz cycles}) admitting branched covers of arbitrary degree and arbitrary (positive) target genus. More precisely, let $\barH_{g/h,d}$ denote the moduli space (\textit{Hurwitz space}) of Harris-Mumford admissible covers $f:X\to Y$ of degree $d$, where $X,Y$ have genus $g,h$, respectively, and let $\phi:\barH_{g/h,d}\to\barM_g$ be the map remembering the curve $X$ (possibly with non-stable components contracted). We review the theory of admissible covers in \S\ref{hm_section} and \S\ref{galois_section}.

We expect the following:

\begin{conjecture}\label{main_conj}
Suppose $h\ge1$ and $d\ge2$. Then, for all sufficiently large $g$ depending on $h$ and $d$, the class $\phi_{*}([\barH_{g/h,d}])\in H^{*}(\barM_g)$ is non-tautological.
\end{conjecture}

Our methods ultimately fall short of proving Conjecture \ref{main_conj} in full in the following two ways. First, we require a mild condition on $d$ (independent of $g,h$) given by the non-vanishing of the $d$-th Fourier coefficient of a certain modular form. Second, in order for the admissible covers appearing in the pullbacks of our Hurwitz cycles by boundary strata to have the desired topological types, we will need to add additional marked points on our covers satisfying the condition that their images are equal. This is analogous to the situation of \cite{gp,vanzelm}, but in contrast we are not able in general to remove all of the marked points for sufficiently large $g$.

Let $\barH_{g/h,d,(m_2)^2(m_d)^d,n}\in H^{*}(\barM_{g,2m_2+dm_d+n})$ be the locus of genus $g$ curves admitting a degree $d$ cover of a genus $h$ curve, with $m_2$ marked pairs and $m_d$ marked $d$-tuples of points with equal image, along with $n$ marked ramification points. More precisely, let $\barH_{g/h,d,m_2+m_d}$ be the Harris-Mumford space parametrizing covers $f:X\to Y$ as in $\barH_{g/h,d}$, with the data of $m_2+m_d$ additional marked points on $Y$ and their pre-images on $X$, and let $\barH_{g/h,d,(m_2)^2(m_d)^d,n}$ be the class obtained by pushing forward the fundamental class by the map remembering $X$ with the desired marked points. (See also \S\ref{hm_section}.) We then have the following.

\begin{thm}\label{main_thm}
Consider the modular form of weight 24
\begin{equation*}
\eta(q)^{48}=q^2\prod_{\ell\ge1}(1-q^{\ell})^{48}=\sum_{d\ge2}a_dq^d
\end{equation*}
and fix $d$ such that $a_d\neq0$. 

Then, the class $\barH_{g/h,d,(m_2)^2(m_d)^d,n}\in H^{*}(\barM_{g,2m_2+dm_d+n})$ is non-tautological in the following cases.
\begin{itemize}
\item \underline{$h=1$}: $g\ge2$ and $g+m_2\ge12$
\item \underline{$h>1$, $d=2$}: $g\ge2h$, $g+m_2\ge 2h+10$, and $m_2\ge1$
\item \underline{$h>1$, $d>2$}: $g\ge d(h-1)+2$, $g+m_2+m_d\ge (2d-3)(h-1)+12$, and $m_d\ge (d-3)(h-1)+1$
\end{itemize}
\end{thm}

The question of non-vanishing of the Fourier coefficients of $\eta(q)^{48}$ appears to be difficult. As to the related question of the non-vanishing of the Ramanujan tau function $\tau(d)$, that is, the Fourier coefficients of $\eta(q)^{24}$, an old conjecture of Lehmer  \cite{lehmer} predicts that $\tau(d)\neq0$; it is known that Lehmer's conjecture holds for $d\lesssim 8\cdot 10^{23}$ \cite{dhz}.

Because, by definition, tautological classes in singular cohomology are the images of tautological classes in Chow, the corresponding Chow classes are also non-tautological. We also obtain immediately a generalization of \cite[Theorem 2]{vanzelm} and \cite[Theorem 3]{gp} to the \textit{open} loci on $\barM_{g,2m_2}$ of $d$-elliptic curves for $d$ arbitrary when $g+m_2=12$, see Corollary \ref{elliptic_open}.

In order to apply the criterion of Graber-Pandharipande to prove Theorem \ref{main_thm}, one needs a sufficiently robust way to compute pullbacks of Hurwitz cycles on $\barM_{g,N}$ to boundary strata. The development of this method initiated in the work of Schmitt-van Zelm \cite{svz} (see also \S\ref{galois_intersection}) for \textit{Galois} Hurwitz cycles, which was incorporated in to a theory of \textit{H-tautological classes} on moduli spaces of admissible Galois covers in \cite{lian_htaut}. In particular, this new framework allows for the intersection of arbitrary (Harris-Mumford) admissible cover cycles with boundary strata, see \cite[\S 6]{lian_htaut}, which we review in \S\ref{hm_boundary_int}.

\subsection{Summary of proof}
The proof of Theorem \ref{main_thm} proceeds by induction on $g$ and $h$ in three main steps. We may reduce to the case $n=0$ (in all three steps, see Lemma \ref{add_ram}) and $m_d=0$ (in steps 1 and 2, see Lemma \ref{d_to_2}).
\begin{itemize}
\item \underline{Step 1 ($h=1, g+m_2=12$):} We pull back the cycle $\barH_{g/1,d,(m_2)^2}\in H^{*}(\barM_{g,2m_2})$ to the boundary stratum $\barM_{1,11}\times\barM_{1,11}$ obtained by gluing $g-1$ pairs of nodes on the two elliptic components together. We find in Lemmas \ref{type_12_contribution} and \ref{type_3_contribution} that the only possibly non-tautological contributions in the pullback come from from pairs of isogenies $X_1\to Y_1$, $X'_1\to Y_1$ of total degree $d$ over a common target.

Then, the contribution from odd classes on $\barM_{1,11}$ is governed by a Hecke-type operator on $H^{11}(\barM_{1,11})$, which we compute using the description of the non-trivial classes in $H^{11}(\barM_{11})$ in terms of the weight 12 cusp form $\eta(q)^{24}$, see Lemma \ref{hecke_lemma}. Therefore, this contribution is non-zero if and only if $a_d\neq0$, and we conclude that $\barH_{g/1,d,(m_2)^2}\in H^{*}(\barM_{g,2m_2})$ for such $d$.

\item \underline{Step 2 ($h=1, g+m_2>12$):} We induct on $m_2$ and $g$ by pulling back $\barH_{g/1,d,(m_2)^2}\in H^{*}(\barM_{g,2m_2})$ to boundary divisors. The induction on $m_2$ is addressed in Lemma \ref{add_pair} by pulling back to a divisor on $\barM_{g,2(m_2+1)}$ of curves with a 2-pointed rational tail, and the induction on $g$ is addressed in \S\ref{g_induction} by pulling back to a divisor on $\barM_{g,2m_2}$ of curves with an elliptic tail.

\item \underline{Step 3 ($h>1$):} We induct on $h$ by pulling back $\barH_{g/h,d,(m_2)^2(m_d)^d}\in H^{*}(\barM_{g,2m_2+dm_d})$ to a boundary stratum of curves with $d$ elliptic tails attached to a spine of genus $g-d$; this is carried out in \S 6. Here, we require the condition that the $d$ attachment nodes appear in the same fiber of an admissible cover, requiring us to introduce the parameter $m_d$.
\end{itemize}


\subsection{Conventions}

We work exclusively over $\bC$. Cohomology groups are taken with rational coefficients except when otherwise noted; we will also need to pass to complex coefficients to study the odd cohomology class $\omega$ on $\barM_{1,11}$ coming from the weight 12 modular form $\eta(q)^{24}$. We will frequently identify homology and cohomology classes in complementary degrees via Poincar\'{e} duality without mention. All curves are assumed projective and connected with only nodes as singularities, except when otherwise noted, and the genus of a curve refers to its arithmetic genus. All moduli spaces are understood to be stacks, rather than coarse spaces.

\subsection{Acknowledgments}

We thank Alessio Cela, Gavril Farkas, Johan de Jong, Dan Petersen, Johannes Schmitt, and Jason van Zelm for useful discussions related to this paper. This project was completed with the support of an NSF Postdoctoral Fellowship, grant DMS-2001976.

\section{Preliminaries}\label{prelim}

\subsection{Tautological classes}\label{taut_review}

We recall the standard definition:

\begin{defn}
The \textbf{tautological ring} is the smallest system of subrings $R^{*}(\barM_{g,n})\subset A^{*}(\barM_{g,n})$ containing all $\psi$ and $\kappa$ classes and closed under pushforwards by all boundary morphisms $\xi_{\Gamma}:\barM_{\Gamma}\to\barM_{g,n}$ (indexed by stable graphs $\Gamma$) and all forgetful morphisms $\pi:\barM_{g,n+1}\to\barM_{g,n}$.

We also denote the image of the tautological ring in singular cohomology under the cycle class map by $RH^{*}(\barM_{g,n})\subset H^{*}(\barM_{g,n})$.
\end{defn}

We will work primarily in singular cohomology. However, the Hurwitz classes we will consider are all algebraic, and after we have proven that they are non-tautological in cohomology, it is immediate by definition that they are also non-tautological in Chow.

Additive generators for the tautological ring may be given in terms of \textit{decorated boundary classes}, which can be intersected explicitly in terms of the combinatorics of dual graphs, see \cite[Appendix A]{gp}. In particular, one obtains the following criterion, which will be our primary tool for detecting non-tautological classes.

\begin{prop}\cite[Proposition 1]{gp}\label{nontaut_criterion}
Suppose that $\alpha\in RH^{*}(\barM_{g,n})$ is a tautological class, and let $\xi_{\Gamma}:\barM_{\Gamma}\to\barM_{g,n}$ be a boundary class. Then, on the space $\barM_{\Gamma}=\prod_{v\in V(\Gamma)}\barM_{g_v,n_v}$, the pullback $\xi_{\Gamma}^{*}\alpha$ has \textit{tautological K\"unneth decomposition (TKD)}, that is,
\begin{equation*}
\xi_{\Gamma}^{*}\alpha\in \bigotimes_{v\in V(\Gamma)}RH^{*}(\barM_{g_v,n_v})\subset H^{*}\left(\prod_{v\in V(\Gamma)}\barM_{g_v,n_v}\right)
\end{equation*}
\end{prop}

In particular, when the K\"unneth decomposition (TKD) of $\xi_{\Gamma}^{*}\alpha$ includes non-trivial contributions from odd cohomology, we may immediately conclude that $\alpha$ is non-tautological, as the tautological ring lives in even degree.

\subsection{Cohomology of $\barM_{1,11}$}\label{m11_intro}

Following \cite{gp,vanzelm}, we will detect classes without TKD via the existence of odd cohomology on $\barM_{1,11}$. Here, we collect the facts that we will need.

\begin{lem}\cite[Corollary 1.2]{petersen}\label{even_algebraic}
All even-dimensional classes (and hence, all algebraic classes) on $\barM_{1,11}$ are tautological, that is, $RH^{*}(\barM_{1,11})=H^{2*}(\barM_{1,11})$.
\end{lem}

\begin{lem}\cite[Lemma 8(i)]{vanzelm}\label{boundary_taut}
All algebraic classes on $\barM_{1,11}\times\barM_{1,11}$ supported on the boundary have TKD.
\end{lem}

\begin{lem}\cite{getzler}\label{odd_cohomology}
The odd cohomology of $\barM_{1,11}$ is two-dimensional, generated by the class of a holomorphic 11-form $\omega\in H^{0}(\barM_{1,11},\Omega^{11}_{\barM_{1,11}})\subset H^{11}(\barM_{1,11})$ and its conjugate.
\end{lem}

One can write down the form $\omega$ explicitly, see \cite[\S 2.3]{handbook}. Complex-analytically, the open locus $\cM_{1,11}$ may be regarded as the open subset of
\begin{equation*}
(\bH\times\bC^{10})/(\SL_2(\bZ)\ltimes (\bZ^2)^{10})
\end{equation*}
obtained by removing the diagonals and zero-sections. In the semi-direct product, $\SL_2(\bZ)$ acts on the factors $\bZ^2$  by via the \textit{conjugate} representation
\begin{equation*}
\begin{bmatrix} a & b \\ c & d \end{bmatrix}\cdot(x,y)=\begin{bmatrix} a & -b \\ -c & d \end{bmatrix}\begin{bmatrix} x \\ y\end{bmatrix}.
\end{equation*}
The group action is given by the formula
\begin{equation*}
\left(\begin{bmatrix} a & b \\ c & d \end{bmatrix},\{(x_i,y_i)\})\right)\cdot(z,\{\zeta_i\})=\left(\frac{az+b}{cz+d},\left\{\frac{\zeta_i}{cz+d}+x_i+y_i\cdot\frac{az+b}{cz+d}\right\}\right)
\end{equation*}

From here, one checks that 
\begin{equation*}
\omega=\eta(e^{2\pi iz})^{24}dz\wedge d\zeta_1\wedge\cdots\wedge d\zeta_{10},
\end{equation*}
where $\eta(e^{2\pi iz})^{24}$ is the normalized discriminant cusp form of weight 12 with Fourier expansion
\begin{equation*}
\eta(q)^{24}=q\prod_{\ell\ge1}(1-q^\ell)^{24},
\end{equation*}
is a non-zero holomorphic 11-form on $\cM_{1,11}$, and moreover extends to $\barM_{1,11}$.

\begin{rem}
The discriminant form $\eta(q)^{24}$ is often denoted simply $\Delta(q)$, but we will avoid this notation, reserving the letter $\Delta$ for the diagonal $\Delta:\barM_{1,11}\to\barM_{1,11}\times\barM_{1,11}$.
\end{rem}

\subsection{Hurwitz spaces and admissible covers}\label{hm_section}

Let $\cH_{g/h,d}$ denote the moduli space (\textit{Hurwitz space}) of degree $d$ simply ramified covers $f:X\to Y$, where $X,Y$ are smooth, connected, and proper curves of genus $g,h$, and let $\barH_{g/h,d}$ be its compactification by Harris-Mumford admissible covers, see \cite{hm}.

Recall that, by definition, the branched points of $Y$ are also marked, and the resulting marked curve is required to be stable. In addition, all pre-images on $X$ of the marked points (including those that are not ramification points) are marked, and the resulting curve is automatically stable. Therefore, we get maps

\begin{equation*}
\xymatrix{
\barH_{g/h,d} \ar[r]^{\phi} \ar[d]^{\delta} & \barM_{g,N} \\
\barM_{h,b} &
}
\end{equation*}
where $b=(2g-2)-d(2h-2)$ and $N=(d-1)b$. We refer to $\phi,\delta$ as the \textit{source} and \textit{target} maps, respectively. 

The target map $\delta$ is quasi-finite and surjective, with degree given by a \textit{Hurwitz number}, counting monodromy actions on the fibers of a degree $d$ cover. In particular, $\barH_{g/h,d}$ has dimension $3h-3+b$. The map $\delta$ is in addition unramified over $\cM_{h,b}$, so that $\cH_{g/h,d}$ is smooth, but in general ramified at any cover ramified over at least one node.

More precisely, let $f:X\to Y$ be an admissible cover, and let $\bC[[t_1,\ldots,t_{3h-3+b}]]$ be the complete local ring of the marked target curve. Suppose further that $t_1,\ldots,t_n$ are smoothing parameters for the nodes $y_1,\ldots,y_n$ of $Y$, and for $i=1,\ldots,n$, let $t_{i,1},\ldots_{i,r_i}$ be smoothing parameters for the corresponding nodes of $X$ above $y_i$, with ramification indices $a_{i,1},\ldots,a_{i,r_i}$. Then, the complete local ring of $\cH_{g/h,d}$ at $[f]$ is
\begin{equation*}
\mathbb{C}\left[\left[t_1,\ldots,t_{3h-3+b},\{t_{i,j}\}^{1\le i\le n}_{1\le j\le r_i}\right]\right]/\left(t_1=t_{1,1}^{a_{1,1}}=\cdots=t_{1,r_1}^{a_{1,r_1}},\ldots,t_n=t_{n,1}^{a_{n,1}}=\cdots=t_{n,r_n}^{a_{n,r_n}}\right).
\end{equation*} 

In particular, $\cH_{g/h,d}$ is Cohen-Macaulay, but singular at any cover with at least one nodal fiber with more than one ramification point.

\subsubsection{$\psi$ classes}

\begin{defn}
For any marked point $x_i$ on a source curve parametrized by $\barH_{g/h,d}$, we define the corresponding $\psi$ class $\psi_{i}\in A^{1}(\barH_{g/h,d})$ simply by the pullback of the corresponding $\psi$ class from $\barM_{g,N}$.
\end{defn}

In fact, the $\psi$ classes of points living in the same fiber are all constant multiples of each other; more precisely, the $\psi$ class at $x\in X$ is equal to the pullback of the $\psi$ class at $f(x)\in Y$ by $\delta$, divided by the ramification index at $x$, cf. \cite[Lemma 3.9]{svz}

\subsubsection{Additional marked points}

We will also need the following variant: 

\begin{defn}
For any $m\ge0$, let $\barH_{g/h,d,m}$ be the space of admissible covers where we mark $m$ points on the target curve $Y$ in addition to the branch points (and still require that the resulting curve be stable), along with their $md$ unramified pre-images.
\end{defn}


\subsubsection{Hurwitz cycles}

The cohomology classes we will be interested in come from pushing forward the fundamental class of $\barH_{g/h,d,m}$ by the source map to $\barM_{g,N}$, then forgetting marked points (and stabilizing) to get a class on $\barM_{g,n}$, for $n\le N$. We will refer to such classes collectively as \textit{Hurwitz cycles}. More precisely, we have the following. 

\begin{defn}
Suppose that $m=m_2+m_d$ for $m_2,m_d\ge0$ and let $n$ be an integer satisfying $0\le n\le b=(2g-2)-d(2h-2)$. Let $\phi':\barH_{g/h,d,m}\to\barM_{g,2m_2+dm_d+n}$ be the map obtained by post-composing the usual source map $\phi$ with the map forgetting all marked points except:
\begin{itemize}
\item 2 points in each of $m_2$ of the marked unramified fibers,
\item al $d$l points in the other $m_d$ marked fibers, and
\item $n$ simple ramification points.
\end{itemize}

We then (abusing notation) define the \textit{Hurwitz cycle} $\barH_{g/h,d,(m_2)^2(m_d)^d,n}$ in $A^{*}(\barM_{g,2m_2+dm_d+n})$ or $H^{*}(\barM_{g,2m_2+dm_d+n})$ by the pushforward of the fundamental class of $\barH_{g/h,d,m,n}$ by $\phi'$.
\end{defn}

Similarly, we may define the open cycles $\cH_{g/h,d,(m_2)^2(m_d)^d,n}$ by pullback of the Hurwitz cycles, as defined above, to $\cM_{g,2m_2+dm_d+n}$. (Note that the source maps $\phi:\cH_{g/h,d}\to\cM_{g,N}$ are in general not proper, so one cannot take this pushforward directly.)

Strictly speaking, one gets different cycles from different choices of points to forget, but these cycles are related by automorphisms permuting the labels of the marked points. In particular, the property of being tautological is agnostic to these choices, so we suppress them. When any of $m_2,m_d,n$ are equal to zero, we may also suppress them from the notation when there is no risk of confusion.

We will refer in the rest of this section, for notational convenience, to spaces $\barH_{g/h,d}$ of admissible covers, but the discussion carries over immediately to the setting where additional marked points are added, or more generally where source curves are allowed to be disconnected and/or with higher ramification.

\subsubsection{Boundary strata}

We now describe boundary strata on $\barH_{g/h,d}$. Suppose that $\Gamma,\Gamma'$ are stable graphs parametrizing boundary strata on $\barM_{g,N},\barM_{h,b}$, respectively.

\begin{defn}
An \textit{admissible cover of stable graphs} $\gamma:\Gamma\to\Gamma'$ of degree $d$ by a collection of maps on vertices, half-edges, and legs, respectively:
\begin{align*}
\gamma_V:V(\Gamma)&\to V(\Gamma')\\
\gamma_H:H(\Gamma)&\to H(\Gamma')\\
\gamma_L:L(\Gamma)&\to L(\Gamma')\\
\end{align*}
compatible (in the obvious sense) with all of the attachment data, in addition to the data of a degree $d_v$ at each $v\in V(\Gamma)$, and a (common) ramification index $d_e$ at each $e\in E(\Gamma)$, such that:
\begin{itemize}
\item if $v\in V(\Gamma)$ and $h'\in H(\Gamma')$ is a half-edge attached to $\gamma_V(v)$, then the sum of the ramification indices at the half-edges attached to $v$ living over $h'$ is equal to $d_v$, and
\item if $v'\in V(\Gamma')$, then the sum of the degrees at vertices living over $v'$ is equal to $d$.
\end{itemize}
\end{defn}

\begin{rem}
The notion of an admissible cover of stable graphs is different from the notion of an \textit{$A$-structure} $A\to\Gamma$ on a stable graph, see \cite[\S A.2]{gp} or \cite[Definition 2.5]{svz} which captures the phenomenon of stable curve with dual graph $A$ degenerating to one of dual graph $\Gamma'$. Either notation could sensibly be referred to as a morphism of stable graphs; we avoid doing so as not to cause confusion.
\end{rem}

Let $\gamma:\Gamma\to\Gamma'$ be a degree $d$ admissible cover of stable graphs. Then, for each $v'\in V(\Gamma')$, let $\barH_{v'}$ be the moduli space of admissible covers of the topological type given by the pre-image of $v'$ in $V(\Gamma)$, along with the data of the attached half-edges and legs. Note that such covers will in general have disconnected targets and arbitrary ramification, but the discussion above applies in this more general setting. We then get a boundary stratum
\begin{equation*}
\xi_{(\Gamma,\Gamma')}:\barH_{(\Gamma,\Gamma')}\to\barH_{g/h,d}
\end{equation*}
by gluing the constituent admissible covers over each component of $Y$ according to the data of $\Gamma$ and $\Gamma'$ (we have suppressed the map $\gamma$ from the notation).

It is clear that the maps $\xi_{(\Gamma,\Gamma')}$ are quasi-finite and that their images cover the boundary of $\barH_{g/h,d}$. The codimension of a boundary stratum is equal to the number of edges of $\Gamma'$, and their specializations to one another can be described in terms of the combinatorics of the admissible covers of graphs (we will not need such an explicit description).

\subsubsection{Separating nodes}

We record here the following straightforward lemma.

\begin{lem}\label{sep_node}
Let $f:X\to Y$ be an admissible cover, and suppose that $x\in X$ is a separating node. Then, $f(x)\in Y$ is a separating node.
\end{lem}

\subsection{Admissible Galois covers}\label{galois_section}

As we have already noted, $\barH_{g/h,d}$ is in general singular at the boundary. This will pose only minor problems for our purposes; in some instances, however, we will need to pass, at least implicitly, to its normalization.

Let $G$ be a finite group. Let $\barH_{g,G,\xi}$ be the moduli space of \textit{admissible $G$-covers} $f:X\to Y$ with monodromy $\xi$, where $X$ is a stable curve of genus $g$ with a generically free $G$-action, and $f$ identifies $Y$ with the scheme-theoretic quotient $X/G$. (See \cite[\S 3]{svz} for detailed definitions.) Recall that we also have source and target maps
\begin{equation*}
\xymatrix{
\barH_{g,G,\xi}\ar[r]^{\phi} \ar[d]^{\delta} & \barM_{g,N} \\
\barM_{h,b}
}
\end{equation*}
where $h$ is the genus of $X/G$.

As in the Harris-Mumford setting, $\delta$ is quasi-finite, and is ramified at $G$-covers with ramification at nodes. However, $\barH_{g,G,\xi}$ is \textit{smooth} of dimension $3h-3+b$, and the map $\phi$ is in addition finite and unramified, see \cite[Theorem 3.7]{svz}.

As in the Harris-Mumford setting, we may define $\psi$ classes on $\barH_{g,G,\xi}$ by pullback by $\phi$ (or, with a correction factor, by $\delta$); here, any two $\psi$ classes at marked points in the same $G$-orbit are equal.

\subsubsection{Boundary strata}

Boundary strata $\xi_{(\Gamma,G)}:\barH_{(\Gamma,G)}\to \barH_{g,G,\xi}$ on $\barH_{g,G,\xi}$ are indexed by \textit{admissible $G$-graphs} $(\Gamma,G)$, see \cite[\S 3.4]{svz}. The space $\barH_{(\Gamma,G)}$ is a product, indexed by the vertices $v$ of the quotient graph $\Gamma/G$, of moduli spaces of admissible $G_v$-covers, where $G_v\subset G$ is the stabilizer of any lift of $v$ to $\Gamma$. However, it will later be convenient to regard these factors equivalently as spaces of disconnected admissible $G$-covers whose components indexed by left cosets of $G_v$ in $G$.

\subsubsection{Normalizing of the Harris-Mumford space}\label{hm_normalization_section} We now explain how $\barH_{g/h,d}$ may be normalized via moduli of admissible Galois covers, see also \cite[\S 1.3, \S 6.1]{lian_htaut}. Let $f:X\to Y$ be a degree $d$ cover of smooth curves. and let $f_0:X_0\to Y_0$ be the \'{e}tale locus. Define
\begin{equation*}
\wt{f}_0:\wt{X}_0:=(X_0\times_{Y_0}\cdots\times_{Y_0} X_0)-\Delta\to Y_0
\end{equation*}
given by taking the $d$-fold product over $X_0$ and removing all diagonals, and define $\wt{f}:\wt{X}_0\to Y$ to be the unique extension of $\wt{f}_0$ to a map of smooth and proper curves. Then, $\wt{f}$ is an $S_d$-Galois cover of smooth curves, and the data of $f$ can be recovered from a $S_d$-cover $\wt{X}\to Y$ by defining $X=\wt{X}/S_{d-1}$. Note, however, that $\wt{X}$ may not be connected.

If, on the other hand, $f$ is an admissible cover, this construction in general does not yield a map of stable curves. It may instead be carried out over the components of $Y$ separately, but there will in general be multiple ways to glue together the resulting maps to form an admissible $S_d$-cover with the property that $X=\wt{X}/S_{d-1}$.

In any case, we obtain a map $\nu:\wt{H}_{g/h,d}:=\barH_{\wt{g},S_d,\xi}\to\barH_{g/h,d}$, for appropriately chosen $\wt{g},\xi$ (note here that $\wt{g}$ will be a vector of integers, corresponding to the fact that the curves $\wt{X}$ may be disconnected), defined by
\begin{equation*}
\nu([\wt{f}:\wt{X}\to Y])=[f:X\to Y].
\end{equation*}

Then, $\nu$ is a normalization: indeed, one may identify $\barH_{\wt{g},S_d,\xi}$ with appropriate components of the Abramovich-Corti-Vistoli space of \textit{twisted $G$-covers}, see \cite[Remark 3.6]{svz}, which normalizes the Harris-Mumford space, see \cite[Proposition 4.2.2]{acv}.

\section{Intersections of Hurwitz cycles with boundary strata}\label{hurwitz_int}

\subsection{The Galois case}\label{galois_intersection}

We first recall the main result of Schmitt-van Zelm concerning the intersection of \textit{Galois} Hurwitz loci with boundary strata on $\barM_{g,N}$. We consider the pullback of 
\begin{equation*}
\phi:\barH_{g,G,\xi}\to\barM_{g,N}
\end{equation*}
by the boundary class $\xi_A:\barM_{A}\to\barM_{g,N}$. 

We have a Cartesian diagram \cite[Proposition 4.3]{svz}
\begin{equation*}
\xymatrix{
\displaystyle\coprod\barH_{(\Gamma,G)} \ar[d]_{\coprod \phi_{\alpha}} \ar[r]^{\xi_{(\Gamma,G)}} & \barH_{g,G,\xi} \ar[d]^{\phi}\\
\barM_{A} \ar[r]^{\xi_{A}} & \barM_{g,n}
}
\end{equation*}
where the coproduct is over admissible $G$-graphs $(\Gamma,G)$ equipped with an $A$-structure $\alpha:\Gamma\to A$ satisfying the genericity condition that the induced map 
\begin{equation*}
\alpha_E:E(A)\to E(\Gamma)/G
\end{equation*}
from edges of $A$ to \textit{$G$-orbits of} edges of $\Gamma$ is surjective. The $A$-structures $\alpha:\Gamma\to A$ then naturally induce the maps $\phi_{\alpha}$ on the left.

The normal bundle of $\barM_A$ in $\barM_{g,r}$ is the direct sum of line bundle contributions from the edges of $A$, and the normal bundle of $\barH_{(\Gamma,G)}$ in $\barH_{g,G,\xi}$ is the direct sum of line bundle contributions of $G$-orbits of edges of $\Gamma$. By the excess intersection formula, we conclude:
\begin{thm}\cite[Theorem 4.9]{svz}\label{svz_thm}
With notation as above, we have
\begin{equation*} 
\xi^{*}_A(\phi_{*}([\barH_{g,G,\xi}]))=\sum_{(\Gamma,G)}\phi_{\alpha*}\left(\prod_{(\ell,\ell')}(-\psi_{\ell}-\psi_{\ell'})\right),
\end{equation*}
where $(\ell,\ell')$ is a pair of half-edges comprising an edge, and we range over edges of in the image of $E(A)\to E(\Gamma)$, excluding a choice of contributions from $G$-orbit representatives of $E(\Gamma)$.
\end{thm}

More generally, if $G_1\subset G$ is any subgroup, one can compute the pullback of the \textit{restriction map} $\res_{G_1}^{G}:\barH_{g,G,\xi}\to\barH_{g,G_1,\xi'}$ by any boundary class $\xi_{(A,G_1)}:\barH_{(A,G_1)}\to\barH_{g,G_1,\xi'}$, see \cite[Proposition 4.13]{lian_htaut}.

\subsection{The Harris-Mumford case}\label{hm_boundary_int}

Now, we consider the analogous question on the Harris-Mumford space $\barH_{g/h,d}$ (or any of its variants). Let $\xi_A:\barM_{A}\to\barM_{g,N}$ be a boundary class as before. 

\begin{prop}\label{hm_cartesian_sets}
We have a commutative diagram
\begin{equation*}
\xymatrix{
\coprod \barH_{(\Gamma,\Gamma')} \ar[r] \ar[d]_{\coprod \phi_{(\Gamma,\Gamma')}} & \barH_{g/h,d} \ar[d]^{\phi} \\
\barM_{A} \ar[r]^{\xi_{A}} & \barM_{g,N}
}
\end{equation*}
where the disjoint union is over boundary strata along with an $A$-structure on $\Gamma$, with the genericity condition that the composite map
\begin{equation*}
E(A)\subset E(\Gamma)\to E(\Gamma')
\end{equation*}
is surjective.

Furthermore, the diagram is Cartesian \textbf{on the level of closed points}.
\end{prop}

\begin{proof}
The commutativity is clear. We construct the inverse map
\begin{equation*}
\barH_{g/h,d} \times_{\barM_{g,N}} \barM_{A}(\Spec(\bC)) \to \coprod \barH_{(\Gamma,\Gamma')}(\Spec(\bC)) 
\end{equation*}

Let $[f:X\to Y]$ be a point of $\barH_{g/h,d}$ with an $A$-structure on the dual graph of $X$. Then, we get a natural stable graph $\Gamma'$ as follows. Let $E(\Gamma')$ be the set of nodes to which the nodes of $X$ corresponding to the edges of $A$ map. Let $V(\Gamma')$ be the set of connected components of the curve obtained by deleting the nodes of $E(\Gamma')$ from $Y$, and let $L(\Gamma')$ be the set of marked points of $Y$; together these define a natural stable graph $\Gamma'$ and a $\Gamma'$-structure on the dual graph of $Y$.

Now, let $E(\Gamma)$ be the set of nodes of $X$ living over $E(\Gamma')$, $V(\Gamma)$ be the set of components of the curve obtained by deleting these nodes from $X$, and $L(\Gamma)$ be the set of marked points of $X$. As before, we get a stable graph $\Gamma$, along with a natural $\Gamma$-structure on the dual graph of $X$ and an $A$-structure on $\Gamma$. The topology of $f$ also induces an admissible cover $\gamma:\Gamma\to\Gamma'$. The genericity condition above is visibly satisfied, and we obtain from $f$ a point of $\coprod \barH_{(\Gamma,\Gamma')}$; it is straightforward to check that we get the desired inverse.
\end{proof}

In general, the diagram in Proposition \ref{hm_cartesian_sets} will fail to be a functorial fiber diagram on the level of stacks owing to non-reduced structure on the intersection of $\phi$ and $\xi_A$. To see this, we compute the local picture.

Let $t_1,\ldots,t_k$ be deformation parameters corresponding to the edges of $\Gamma'$ (which in turn correspond to smoothing parameters of nodes of $Y$), and let $t_{i,j}$ be deformation parameters for the edges of $\Gamma$ living above $t_i$ (which in turn correspond to smoothing parameters of nodes of $X$). Let $t_{k+1},\ldots,t_{3h-3+b}$ be deformation parameters for $Y$ away from the chosen nodes. Then, recall from \S\ref{hm_section} that the complete local ring of $\barH_{g/h,d}$ at $[f]$ may be written as
\begin{equation*}
\bC[\{t_i\},\{t_{ij}\}]/(t_i=t_{ij}^{a_{ij}})\otimes S,
\end{equation*}
where $a_{ij}$ are the associated ramification indices and $S$ is the complete local ring of $\coprod \barH_{(\Gamma,\Gamma')}$ at the point obtained by deleting all of the nodes of $X$ and $Y$ corresponding to the edges of $\Gamma$ and $\Gamma'$.

The effect of pulling back by $\xi_{A}$, on the level of complete local rings, kills all smoothing parameters corresponding to the nodes of $A$. In particular, all of the variables $t_i$ are killed, and we are left with the in ring
\begin{equation*}
R_{(\Gamma,\Gamma')}:=\bC[\{t_{ij}\}]/(t_{ij}^{a_{ij}})\otimes S,
\end{equation*}
where we now range over all $(i,j)$ not corresponding to an edge of $A$.

In general, the complete local ring $\Spec(R_{(\Gamma,\Gamma')})$ is non-reduced, in which case the functorial intersection of $\phi$ and $\xi_A$ is non-reduced with underlying reduced space $\coprod \barH_{(\Gamma,\Gamma')}$. Each boundary stratum $\barH_{(\Gamma,\Gamma')}$ contributes separately to the class $\xi_A^{*}\phi_{*}([\barH_{g/h,d}])$; we now explain how to compute this contribution. 

\begin{prop}\label{harris_mumford_contributions}
With notation as above, consider the contribution of $\barH_{(\Gamma,\Gamma')}$ to the class $\xi_A^{*}\phi_{*}([\barH_{g/h,d}])$
\begin{enumerate}
\item[(a)] Suppose that $\barH_{(\Gamma,\Gamma')}$ has the expected dimension. Then, its contribution to $\xi_A^{*}\phi_{*}([\barH_{g/h,d}])$ is a non-zero multiple of $\phi_{*}([\barH_{(\Gamma,\Gamma')}])$ on the boundary stratum $\barM_A$,
\item[(b)] If $\barH_{(\Gamma,\Gamma')}$ is arbitrary, its contribution to $\xi_A^{*}\phi_{*}([\barH_{g/h,d}])$ is the pushforward by $\phi_{(\Gamma,\Gamma')}$ by a polynomial in $\psi$ classes on $\barH_{(\Gamma,\Gamma')}$ at half-edges of $\Gamma$ (capped against the fundamental class of $\barH_{(\Gamma,\Gamma')}$)
\end{enumerate}
\end{prop}

\begin{proof}
If $\barH_{(\Gamma,\Gamma')}$ has the expected dimension, then its contribution to the intersection is the fundamental class of a union of components with underlying reduced $\barH_{(\Gamma,\Gamma')}$ and multiplicity equal to the length of $R_{(\Gamma,\Gamma')}$, by the above discussion. This gives part (a).

For part (b), we apply the excess intersection formula. Morally, the situation is analogous to the Galois case, but because the spaces in question are singular, we will need to pass to their normalizations as in \S\ref{hm_normalization_section}.

Recall that the composite map $\wt{\phi}:\wt{H}_{g/h,d}\to\barM_{g,d}$ may be factored as the composition of the restriction map $\res^{S_d}_{S_{d-1}}:\barH_{\wt{g},S_d,\xi}\to \barH_{\wt{g},S_{d-1},\xi'}$ and the target map $\delta:\barH_{\wt{g},S_{d-1},\xi'}\to\barM_{g,N}$. 

Consider the pullback of $\xi_{A}$ first by $\delta$. By \cite[\S 4.3.2]{lian_htaut}, the result is a disjoint union of boundary strata on $\barH_{\wt{g},S_{d-1},\xi'}$, all of the expected dimension (equal to that of $\barM_A$), each appearing with multiplicity given in terms of the ramification indices appearing, see \cite[Lemma 4.15]{lian_htaut}. We may then pull back the underlying reduced spaces (the boundary strata themselves) by the restriction map, to obtain a disjoint union of boundary strata $\barH_{(\wt{\Gamma},S_d)}$ on $\barH_{\wt{g},S_d,\xi}$, as in \cite[Lemma 4.11]{lian_htaut}.

By \cite[Proposition 4.13]{lian_htaut} (that is, the analogue of Theorem \ref{svz_thm} in the H-tautological setting), the class $\xi_A^{*}\phi_{*}([\barH_{g/h,d}])$ is then computed in terms of $\psi$ classes on $\barH_{(\wt{\Gamma},S_d)}$ associated to the half-edge (orbits) of $\wt{\Gamma}$, after re-introducing the correction factors of the multiplicities of the boundary classes on $\barH_{\wt{g},S_{d-1},\xi'}$.

On the other hand, the union of the $\barH_{(\wt{\Gamma},S_d)}$ is the underlying reduced space of the pullback of $\xi_A$ by $\wt{\phi}$, so admits a natural map (compatible with the maps to $\barM_A$) 
\begin{equation*}
\nu_{(\Gamma,\Gamma')}:\coprod \barH_{(\wt{\Gamma},S_d)}\to \coprod \barH_{(\Gamma,\Gamma')}.
\end{equation*}
In fact, one can easily make this map explicit: we have
\begin{equation*}
(\Gamma,\Gamma')=(\wt{\Gamma}/S_{d-1},\wt{\Gamma}/S_{d}),
\end{equation*}
the admissible cover $\gamma:\Gamma\to\Gamma':$ is the natural quotient map, and $\nu$ sends $\wt{X}'\to Y'$ to $\wt{X}'/S_{d-1}\to Y'$ over each component $Y'\subset Y$. In particular, above each $\barH_{(\Gamma,\Gamma')}$, the map $\nu_{(\Gamma,\Gamma')}$ is a union of copies of the normalizations of the constituent spaces. These copies are indexed by possible ways of gluing the Galois closures of the individual components of the covers appearing in $\cH_{(\Gamma,\Gamma')}$, or equivalently, by branches of the image of $\coprod \barH_{(\Gamma,\Gamma')}$ in $\barH_{(g/h,d)}$ before normalization, cf. \cite[\S 6.2, step (ii)]{lian_htaut}.

Finally, the $\psi$ classes occurring on $\barH_{(\wt{\Gamma},S_d)}$ may be identified (up to appropriate constant factors) with those on $\coprod \barH_{(\Gamma,\Gamma')}$ via the normalization map, so we may express the contribution from $\barH_{(\Gamma,\Gamma')}$ to $\xi_A^{*}\phi_{*}([\barH_{g/h,d}])$ in the desired way.
\end{proof}

\subsection{Hurwitz cycles with rational target}

We will later need the following statements which identify, in contrast with our main results, tautological contributions to pullbacks of Hurwitz cycles by boundary strata. As usual, both results hold true for all of our variants of $\barH_{g/h,d}$ (allowing any combination of additional marked points, higher ramification, or disconnected source curves).

\begin{lem}\label{taut_rational_target}
Consider $\barH_{g/h,d}$ and $\xi_A$ as above, and suppose further that $\barH_{(\Gamma,\Gamma')}$ is a boundary stratum appearing in the fiber product for which all vertices of $\Gamma'$ have genus 0. Then, the contribution of $\barH_{(\Gamma,\Gamma')}$ to $\xi_A^{*}\phi_{*}([\barH_{g/h,d}])$ has TKD on $\barM_A$.
\end{lem}

\begin{proof}
By Proposition \ref{harris_mumford_contributions}(b), this contribution is a polynomial in $\psi$ classes on $\barH_{(\Gamma,\Gamma')}$ at half-edges of $\Gamma$, pushed forward to $\barM_A$. However, because the target genera are all 0, we may identify the components of $\barH_{(\Gamma,\Gamma')}$ with spaces of relative stable maps to $\bP^1$ and the $\psi$ classes on them with Gromov-Witten classes, see \cite[\S 0.2.3, \S 1.2.2]{fp}. The claim is then immediate from \cite[Theorem 2]{fp}.
\end{proof}

\begin{lem}\label{taut_fixed_target}
Let $\delta':\barH_{g/h,d}\to \barM_{h,k}$ be the composition of a target map $\delta$ with a map forgetting any number of marked points. Let $[Y]$ be a point of $\barM_{h,k}$ and $\barH_{g/h,d}(Y)=\delta'^{*}([Y])$. Then, the class of the pushforward of $\barH_{g/h,d}(Y)$ to $\prod_{i}\barM_{g_i,n_i}$ has TKD.
\end{lem}

\begin{proof}
Points of $\barM_{h,k}$ are homologous to each other, so we may assume that $Y$ is a stable marked curve with only rational components. Then, $\barH_{g/h,d}(Y)$ may be expressed as a disjoint union of boundary strata (of the correct dimension) $\barH_{(\Gamma,\Gamma')}$ appearing with multiplicity, for example, by a straightforward analogue of \cite[\S 4.3.2]{lian_htaut} for Harris-Mumford spaces. The claim then follows from \cite[Theorem 2]{fp}.
\end{proof}

\subsection{Post-composing with forgetful maps}\label{forgetful_cartesian}

The results above concern classes coming from source maps $\phi:\barH_{g/h,d}\to\barM_{g,N}$, but we will be primarily concerned with classes obtained by post-composing with forgetful maps $\pi:\barM_{g,N}\to\barM_{g,r}$. The situation here is similar: we need only note that we have a Cartesian diagram (with the intersection occurring in the correct dimension)
\begin{equation*}
\xymatrix{
\coprod\barM_{g,A'} \ar[r]^{\coprod\xi_{A'}} \ar[d]_{\coprod\pi} & \barM_{g,N} \ar[d]^{\pi} \\
\barM_{g,A} \ar[r]^{\xi_A} & \barM_{g,r}
}
\end{equation*}
where the coproduct is over stable graphs $A'$ obtained from $A$ by distributing the remaining points over its vertices in all possible ways.

\section{Reductions}\label{reductions}

\begin{lem}[{cf. \cite[Lemma 10]{vanzelm}}]\label{add_ram}
Suppose that $\barH_{g/h,d,(m_2)^2(m_d)^d}\in H^{*}(\barM_{g,2m_2+dm_d})$ is non-tautological. Then, $\barH_{g/h,d,(m_2)^2(m_d)^d,n}\in H^{*}(\barM_{g,2m_2+dm_d+n})$ is non-tautological for all $n\ge0$.
\end{lem}

\begin{proof}
Up to a non-zero constant, the class $\barH_{g/h,d,(m_2)^2(m_d)^d,n}\in H^{*}(\barM_{g,2m_2+dm_d+n})$ pushes forward to $\barH_{g/h,d,(m_2)^2(m_d)^d}\in H^{*}(\barM_{g,2m_2+dm_d})$ upon forgetting the ramification points, so the result is immediate from the fact that tautological classes are closed under forgetful pushforwards.
\end{proof}

Lemma \ref{add_ram} immediately reduces Theorem \ref{main_thm} to the case $n=0$; we assume this henceforth unless otherwise mentioned.

\begin{lem}\label{d_to_2}
Suppose that $m_2\ge1$, and that $\barH_{g/h,d,(m_2)^2(m_d)^d}\in H^{*}(\barM_{g,2m_2+dm_d})$ is non-tautological. Then, $\barH_{g/h,d,(m_2-1)^2(m_d+1)^d}\in H^{*}(\barM_{g,2(m_2-1)+d(m_d+1)})$ is non-tautological.
\end{lem}

\begin{proof}
Up to a non-zero constant, the class $\barH_{g/h,d,(m_2-1)^2(m_d+1)^d}\in H^{*}(\barM_{g,2(m_2-1)+d(m_d+1)})$ pushes forward to $\barH_{g/h,d,(m_2)^2(m_d)^d}\in H^{*}(\barM_{g,2m_2+dm_d})$, so we conclude as in Lemma \ref{add_ram}.
\end{proof}

\begin{lem}[{cf. \cite[Lemma 11]{vanzelm}}]\label{add_pair}
Suppose that $\barH_{g/h,d,(m_2)^2(m_d)^d}\in H^{*}(\barM_{g,2m_2+dm_d})$ is non-tautological. Then, $\barH_{g/h,d,(m_2+1)^2(m_d)^d}\in H^{*}(\barM_{g,2(m_2+1)+dm_d})$ is non-tautological.
\end{lem}

\begin{proof}
We pull back to the boundary divisor
\begin{equation*}
\xi:\barM_{g,2m_2+dm_d+1}\cong\barM_{g,2m_2+dm_d+1}\times\overline{M}_{0,3}\to\barM_{g,2(m_2+1)+dm_d}
\end{equation*}
parametrizing marked curves with a rational tail marked by a pair of unramified points in the same fiber, and apply Proposition \ref{nontaut_criterion}. Let $m=(m_2+1)+m_d$. We have a diagram 
\begin{equation*}
\xymatrix{
\coprod \barH_{(\Gamma,\Gamma')} \ar[r] \ar[d] &\barH_{g/h,d,m} \ar[d]^{\phi}\\
\coprod \barM_{A} \ar[r] \ar[d] & \barM_{g,r} \ar[d]^{\pi}\\
\barM_{g,2m_2+dm_d+1}\times\overline{M}_{0,3} \ar[r]^(0.55){\xi} &\barM_{g,2(m_2+1)+dm_d}.
}
\end{equation*}
where $r=(2g-2)-d(2h-2)+2(m_2+1)+dm_d$, and the stable graphs $A$ arise from all possible ways to distribute the remaining marked points on the two components parametrized by $\barM_{g,2m_2+dm_d+1}\times\overline{M}_{0,3}$ as in \S\ref{forgetful_cartesian}. The bottom square is Cartesian, and the top square is Cartesian at least on the level of closed points as in Proposition \ref{hm_cartesian_sets}; it will turn out that the only strata $\barH_{(\Gamma,\Gamma)}$ contributing to the class in question will, in fact, be reduced.

Consider a general point $[f:X\to Y]$ on some $\barH_{(\Gamma,\Gamma')}$ in the fiber product. Because the graphs $A$ have only one edge, by the genericity condition from Proposition \ref{hm_cartesian_sets}, $Y$ may only have one node, which must be separating by Lemma \ref{sep_node}. 

Therefore, $Y$ is the union of two components $Y_0,Y_h$ of genus $0,h$, respectively. Moreover, if $X$ is the union $X_0\cup X_g$, with the two pieces corresponding to the vertices of $\Gamma$, then $X_0$ must be an irreducible (and smooth) rational curve living entirely over $Y_0$. Note, in addition, that $X_0$ must have degree at least 2 over $Y_0$, in order to contain two marked points in the same fiber.

Let $b=(2g-2)-d(2h-2)+(m_2+1)+m_d$ be the total number of marked points on $Y$, and let $B=b+(3h-3)$ be the dimension of $\barH_{g/h,d,(m_2+1)^2(m_d)^d}$. In order for $\barH_{(\Gamma,\Gamma')}$ to give a non-zero contribution to $H_{2(B-1)}(\barM_{g,2m_2+dm_d+1})$, we need the image of $\barH_{(\Gamma,\Gamma')}$ in $\barM_{g,2m_2+dm_d+1}$ to be supported in dimension (at least) $B-1$.

Let $b_0$ be the number of marked points of $Y$ whose marked pre-images, only including those not forgotten by $\pi$, lie entirely on $X_0$, and let $b_g=b-b_0$ be the number that have at least one marked pre-image on $X_g$. Note that $b_0\ge 2$, with at least one point coming from the unramified pair, and at least one more coming from a ramification point, as the degree of $X_0$ over $Y_0$ is at least 2. Therefore, $b_g\le b-2$. Then, by the quasi-finiteness of the target maps $\delta$, the dimension of the image of $\barH_{(\Gamma,\Gamma')}$ in $\barM_{g,2m_2+dm_d+1}$ is at most $b_g+1+(3h-3)\le B-1$, and that this number decreases if one or more of the $b_g$ marked points with a pre-image on $X_g$ lies on $Y_0$. 

Therefore, we must have equality everywhere. In particular, $X_0$ has degree 2 over $Y_0$ and is ramified over the node of $Y$, $X_g$ consists of a smooth component of genus $g$ mapping with degree $d$ to $Y_h$, ramified at one point over the node of $Y$, along with $d-2$ rational tails mapping isomorphically to $Y_0$. In addition, all marked unramified fibers must lie on $X_g$. The contributing covers are shown in Figure \ref{Fig:m2_induction}.

\begin{figure}[!htb]
     \includegraphics[width=.45\linewidth]{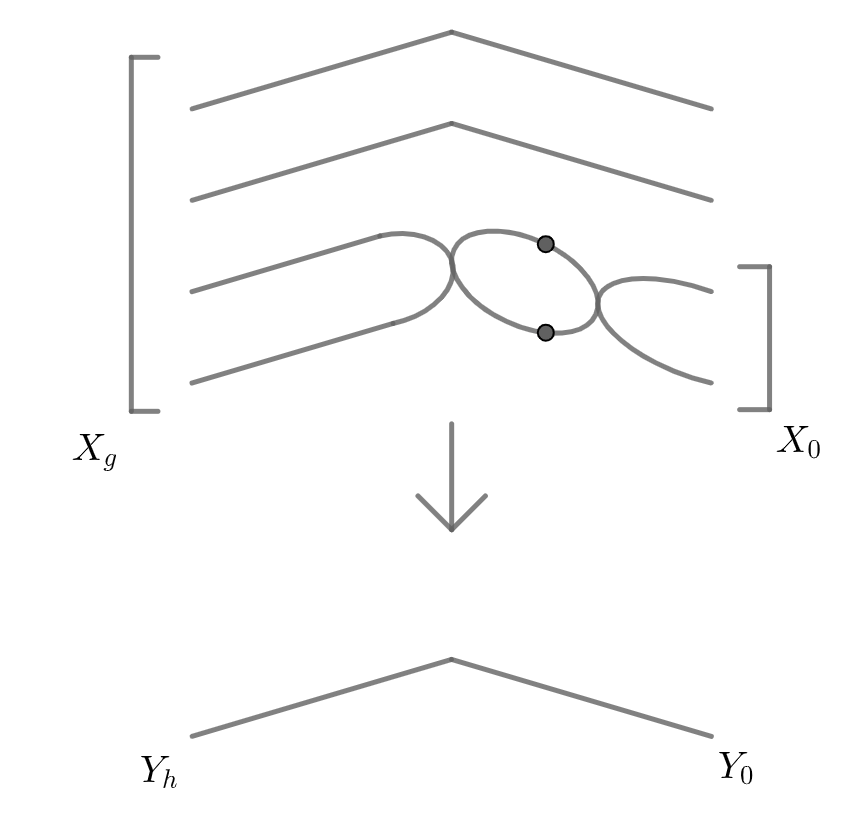}
     \caption{The only possible contribution to $\xi^{*}\barH_{g/h,d,(m_2+1)^2(m_d)^d}$. All other marked points lie on $X_g$.}\label{Fig:m2_induction}
\end{figure}

Now, we see that the pullback of $\barH_{g/h,d,(m_2+1)^2(m_d)^d}\in H^{*}(\barM_{g,2(m_2+1)+dm_d})$ by $\xi$, after forgetting the factor $\overline{M}_{0,3}$, gives, up to a non-zero constant, the class $\barH_{g/h,d,(m_2)^2(m_d)^d,1}\in H^{*}(\barM_{g,2m_2+dm_d+1})$. The proof is now complete by Lemma \ref{add_ram}.
\end{proof}

\section{$d$-elliptic loci}\label{d-ell_section}

In this section, we prove the first part of Theorem \ref{main_thm}, in the case $h=1$. We follow the approach of \cite{vanzelm}: we first handle the case $g+m_2=12$ by finding a non-zero contribution from odd cohomology on $\barM_{1,11}$ upon a boundary pullback, then use a different boundary pullback to induct on $g$.

Recall that we define the integers $a_d$, $d\ge2$ by
\begin{equation*}
\eta(q)^{48}=q^2\prod_{\ell\ge1}(1-q^{\ell})^{48}=\sum_{d\ge2}a_dq^d
\end{equation*}

\subsection{The case $g+m_2=12$}\label{equal12_section}

\begin{prop}\label{elliptic_equal12}
Suppose that $d\ge2$, $g\ge2$, $g+m_2=12$, and $a_d\neq0$. Then, $\barH_{g/1,d,(m_2)^2}\in H^{22}(\barM_{g,2m_2})$ is non-tautological.
\end{prop}

We will prove Proposition \ref{elliptic_equal12} by pullback to the boundary stratum $\xi:\barM_{1,11}\times\barM_{1,11}\to \barM_{g,2m_2}$ defined by gluing $g-1$ pairs of marked points on the two elliptic components together. We have a diagram

\begin{equation*}
\xymatrix{
\coprod \barH_{(\Gamma,\Gamma')} \ar[r] \ar[d] & \barH_{g/1,d,m_2} \ar[d]^{\phi}\\
\coprod \barM_{A} \ar[r] \ar[d] & \barM_{g,N} \ar[d]^{\pi}\\
\barM_{1,11}\times\barM_{1,11} \ar[r]^(0.59){\xi} & \barM_{g,2m_2}
}
\end{equation*}
where $N=(d-1)(2g-2)+dm_2$, and the stable graphs $A$ arise from all possible ways to distribute the remaining marked points on the two components parametrized by $\barM_{1,11}\times\barM_{1,11}$. The bottom square is Cartesian, and the top square is Cartesian on the level of closed points.

We consider the contributions to the intersection of $\xi$ and $\phi':\barH_{g/1,d,m_2}\to\barM_{g,2m_2}$ from each $\barH_{(\Gamma,\Gamma')}$ separately. First, note that if the image of $\barH_{(\Gamma,\Gamma')}$ in $\barM_{1,11}\times\barM_{1,11}$ is supported on the boundary, then the corresponding contribution to $\xi^{*}\barH_{g/1,d,(m_2)^2}$ automatically has TKD by Lemma \ref{boundary_taut}.

Therefore, we can assume in particular that the generic cover $[X\to Y]$ of $\cH_{(\Gamma,\Gamma')}$ has the property that the source curve $X$ has two smooth genus 1 components $X_1,X'_1$. Let $Y_1,Y'_1\subset Y$ be the image components of $X_1,X'_1$, respectively. We then have three cases:
\begin{enumerate}
\item $Y_1\neq Y'_1$,
\item $Y_1=Y'_1$ is a smooth rational curve, and
\item $Y_1=Y'_1$ is a smooth genus 1 curve.
\end{enumerate}

\begin{lem}\label{type_12_contribution}
The contributions to $\xi^{*}\barH_{g/1,d,(m_2)^2}$ from strata $\barH_{(\Gamma,\Gamma')}$ whose general point satisfies either (i) or (ii) have TKD.
\end{lem}

\begin{proof}
First, consider case (i). The component $\barH_{(\Gamma,\Gamma')}$ in question has the property that $\Gamma$ has two vertices of genus 1, which map to different vertices $v_1,v'_1\in\Gamma'$, and the rest of the vertices of $\Gamma$ have genus 0. We may decompose 
\begin{equation*}
\barM_{A}=\barM_{v_1}\times\barM_{v'_1}\times\barM_{w},
\end{equation*}
where $\barM_{v_1},\barM_{v'_1}$ parametrize the components of $X$ mapping to $Y_1,Y'_1$, respectively, and $\barM_{w}$ parametrizes all other components. The spaces $\barM_{v_1},\barM_{v'_1}$ are products of a single moduli space of pointed genus 1 curves with a collection of moduli spaces of pointed rational curves, whereas $\barM_{w}$ is a product of moduli spaces of pointed rational curves. 

The pushforward of $\barH_{(\Gamma,\Gamma')}$ to $\barM_{A}$ decomposes into a product of algebraic classes on the components $\barM_{v_1},\barM_{v'_1},\barM_{w}$, and therefore has TKD, by Lemma \ref{even_algebraic} and the fact that all cohomology on moduli spaces of pointed rational curves is tautological \cite{keel}. In particular, the further pushforward to $\barM_{1,11}\times\barM_{1,11}$ also has TKD.

Now, consider case (ii). Note that all components of $Y$ must be rational, because the only two components of $X$ which can map to a higher genus curve, namely $X_1$ and $X'_1$, both map to a rational component. Therefore, the resulting contribution of $\barH_{(\Gamma,\Gamma')}$ to $\xi^{*}\barH_{g/1,d,(m_2)^2}$ has TKD by Lemma \ref{taut_rational_target}.
\end{proof}

\begin{lem}\label{type_3_contribution}
All strata $\barH_{(\Gamma,\Gamma')}$ whose general point satisfies (iii) and which give non-zero contributions to $\xi^{*}\barH_{g/1,d,(m_2)^2}$ have general point $[f:X\to Y]$ of the following form, also depicted in Figure \ref{Fig:isogeny}.

$Y$ consists of an elliptic component $Y_1$ with $m_2$ marked points and $g-1$ rational tails, each of which contains two branch points. $X$ contains two elliptic components $X_1,X'_1$ over $Y_1$, connected by $g-1$ rational bridges, mapping to the rational tails of $Y$ with degree 2, and all other components living over the rational tails have degree 1. Finally, the unramified pairs of marked points of $X$ live over those of $Y$, with one point of each pair on $X_1$ and $X'_1$. 
\end{lem}

\begin{figure}[!htb]
     \includegraphics[width=.45\linewidth]{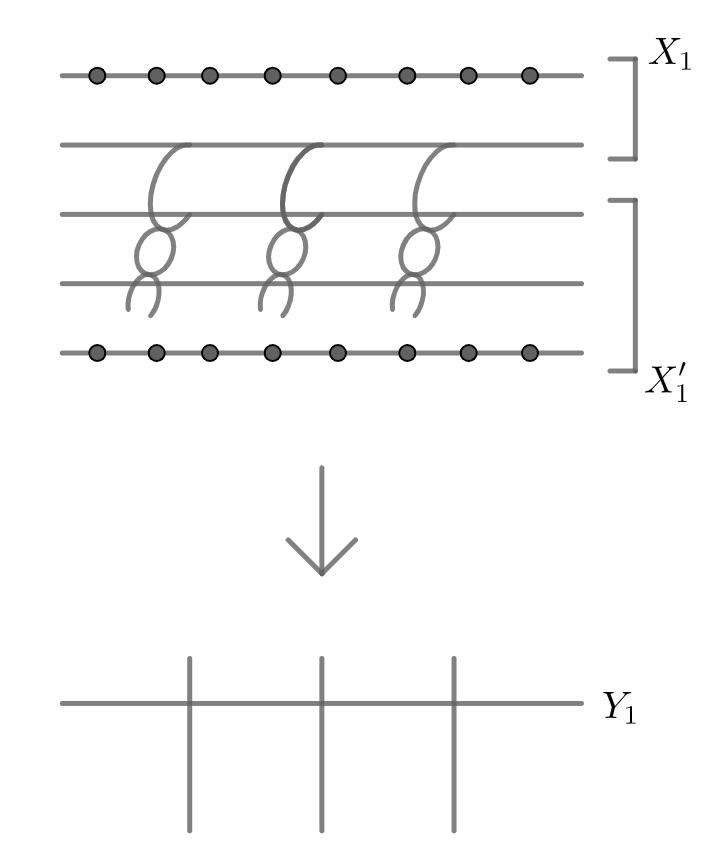}
          \caption{The only possible non-tautological contributions to $\xi^{*}\barH_{g/1,d,(m_2)^2}$. Here, $(g,m_2)=(4,8)$. The rational tails of $X$ mapping isomorphically to those of $Y$ are not shown.}\label{Fig:isogeny}
\end{figure}

\begin{proof}
In order for $\barH_{(\Gamma,\Gamma')}$ to give a non-zero contribution to $\xi^{*}\barH_{g/1,d,(m_2)^2}$, its image in $\barM_{1,11}\times\barM_{1,11}$ must have dimension at least 11. By assumption, the image of $\barH_{(\Gamma,\Gamma')}$ is not supported in the boundary of $\barM_{1,11}\times\barM_{1,11}$, so all of the moduli must live over $Y_1=Y'_1$. Thus, the total number of nodes and marked points on $Y_1$ must be at least 11.

The pre-image of $Y_1$ must consist exactly of the two components $X_1,X'_1$, covering $Y_1$ via isogenies of degrees $d_1,d'_1$, with $d_1+d'_1=d$. In particular, the $s$ marked points on $Y_1$ each correspond to one of the $m_2$ unramified fibers. On the other hand, if there are $t$ nodes on $Y_1$, at which trees of rational components are attached, each such node contributes at least 2 branch points to $Y_1$. Therefore, we have
\begin{equation*}
11\le s+t\le m_2+g-1=11,
\end{equation*}
meaning we have equality everywhere. The conclusion then follows easily.
\end{proof}

To show that, in total, such $\barH_{(\Gamma',\Gamma)}$ give a contribution to $\barM_{1,11}\times\barM_{1,11}$ without TKD, we need the following lemma.

\begin{lem}\label{hecke_lemma}
Let $\barH_{1/1,k,11}^\circ$ be the space of 11-pointed admissible degree $m$ covers $f:X\to Y$, where $X,Y$ have genus 1, and 11 marked points of $X$ are chosen over those of $Y$. (Note that this differs from the usual space $\barH_{1/1,k,11}$ in that here we only mark one point in each fiber.)

Consider the operator $T_k=\phi_{*}\circ\delta^{*}$ acting on $H^{11}(\barM_{1,11})$, induced by the correspondence
\begin{equation*}
\xymatrix{
\barH_{1/1,k}^{11} \ar[r]^{\delta} \ar[d]^{\phi} & \barM_{1,11} \\
\barM_{1,11} &
}.
\end{equation*}
Then, $T_k$ acts on the two-dimensional vector space $H^{11}(\barM_{1,11},\bQ)$ by multiplication by $\tau(k)$, the $q^k$-coefficient of the normalized weight 12 cusp form $\eta(q)^{24}$.
\end{lem}

\begin{proof}
In fact, it suffices to consider the action of $T_k$ on the class of the discriminant form $[\omega]\in H^{11,0}(\cM_{1,11},\bC)$, see \S\ref{m11_intro}. Indeed, $T_k$ necessarily acts by the same constant on both $H^{11}(\barM_{1,11})$ and $H^{11}(\cM_{1,11})$.

We give complex-analytic descriptions of the spaces involved. First, we have $\cM_{1,1}=\bH/\SL_2(\bZ)$. Now, consider $\cH_{1/1,k}$. If $E=\bC/\Lambda$ is an elliptic curve, then isogenies $E'\to E$ are in bijection with index $k$ sublattices $\Lambda\subset\Lambda'$, which in turn are in bijection with the right orbit space $SL_{2}(\bZ)\backslash M_k$, where $M_k$ is the set of integer matrices of determinant $k$. In addition, we have a monodromy action of $SL_2(\bZ)$ on such lattices on the left, and components of $\cH_{1/1,k}$ are indexed by the double orbit space $\SL_2(\bZ)\backslash M_k/\SL_2(\bZ)$. 

Now, for any orbit representative $A\in \SL_2(\bZ)\backslash M_k/\SL_2(\bZ)$, define the congruence subgroup $\Gamma_A\subset\SL_2(\bZ)$ to be the kernel of the left action of $SL_{2}(\bZ)$ on the lattice corresponding to $A\cdot SL_{2}(\bZ)$. We have that $\cH_{1/1,k}$ is the union of modular curves 
\begin{equation*}
\coprod_{A\in \SL_2(\bZ)\backslash M_k/\SL_2(\bZ)} \bH/\Gamma_A
\end{equation*}
where the index set is over a choice of double coset representatives. If $A=\begin{bmatrix}a & b \\ c & d\end{bmatrix}\in M_k$, then the point $z\in \bH/\Gamma_A$ corresponds to the isogeny
\begin{equation*}
\bC/\langle 1,z\rangle \to \bC/\langle cz+d,az+b\rangle \to  \bC/\left\langle 1,\begin{bmatrix}a & b \\ c & d\end{bmatrix}z\right\rangle
\end{equation*} 
where the first map is multiplication by $k$, and the second is the isomorphism induced by multiplication by $\frac{1}{cz+d}$.

In particular, the source map $\phi:\cH_{1/1,k}\to\cM_{1,1}$ is induced by the inclusions $\Gamma_A\subset \Gamma$, so that $\phi(z)=z$, and the target map is defined by $\delta(z)=\begin{bmatrix}a & b \\ c & d\end{bmatrix}z$.

Now, we add marked points: recall from \S\ref{m11_intro} that
\begin{equation*}
\cM_{1,11}\subset (\bH\times\bC^{10})/(\SL_2(\bZ)\ltimes (\bZ^2)^{10}),
\end{equation*}
where the 10 copies of $\bC/\bZ^2$ correspond to the marked points, and the open subset is given by removing the diagonals and zero sections. In a similar way, we have
\begin{equation*}
\barH_{1/1,k,11}^\circ\subset \coprod_{A\in \SL_2(\bZ)\backslash M_k/\SL_2(\bZ)} (\bH\times\bC^{10})/(\Gamma_A\ltimes (\bZ^2)^{10}),
\end{equation*}
with source and target maps are given by 
\begin{align*}
\phi((z,\zeta_1,\ldots,\zeta_{10}))&=(z,\zeta_1,\ldots,\zeta_{10})\\
\delta((z,\zeta_1,\ldots,\zeta_{10}))&=\left(\begin{bmatrix}a & b \\ c & d\end{bmatrix}z,\frac{k\zeta_1}{cz+d},\ldots,\frac{k\zeta_{10}}{cz+d}\right)\\
\end{align*}

We now compute the action of $T_k$ on 
\begin{equation*}
\omega=\eta(z)^{24}dz\wedge d\zeta_1\wedge\cdots\wedge dz_{10}.
\end{equation*}
On $\bH/\Gamma_A$, we have
\begin{align*}
\delta^{*}\omega&=\eta\left(\frac{az+b}{cz+d}\right)^{24}d\left(\frac{az+b}{cz+d}\right)\wedge d\left(\frac{k\zeta_1}{cz+d}\right)\wedge\cdots\wedge d\left(\frac{k\zeta_{10}}{cz+d}\right)\\
&=\eta\left(\frac{az+b}{cz+d}\right)^{24}\left(\frac{k}{(cz+d)^2}dz\right)\wedge\left( \frac{k}{cz+d}d\zeta_1\right)\wedge\cdots\wedge\left( \frac{k}{cz+d}d\zeta_{10}\right)\\
&=k^{11}(cz+d)^{-12}\cdot\eta\left(\frac{az+b}{cz+d}\right)^{24}dz\wedge d\zeta_1\wedge\cdots\wedge dz_{10}.
\end{align*}

To compute the pushforward by $\phi$, recall that the pre-images of a point of $\barM_{1,1}$ are indexed by orbit representatives $A\in \SL_2(\bZ)\backslash M_k$; for each corresponding point of $\barH^{\circ}_{1/1,k}$, we may compute $\delta^{*}(\omega)$ at that point in terms of the chosen matrix $A$. Thus, summing over all pre-images amounts to summing the above formula for $\delta^{*}\omega$ over all choices of $A\in \SL_2(\bZ)\backslash M_k$, and we obtain
\begin{equation*}
\phi_{*}\delta^{*}\omega=k^{11}\sum_{A\in  \SL_2(\bZ)\backslash M_k}(cz+d)^{-12}\eta\left(\frac{az+b}{cz+d}\right)^{24}dz\wedge d\zeta_1\wedge\cdots\wedge dz_{10}.
\end{equation*}
This identifies $T_k$ with the $k$-th Hecke operator on the space of weight 12 cusp forms, which is 1-dimensional, and thus acts by the $k$-th Fourier coefficient of $\eta(q)^{24}$.
\end{proof}

\begin{proof}[Proof of Proposition \ref{elliptic_equal12}]
We wish to show that $\xi^{*}\barH_{g/1,d,(m_2)^2}$ fails to have TKD on $\barM_{1,11}\times\barM_{1,11}$. By Lemmas \ref{type_12_contribution} and \ref{type_3_contribution}, we need only consider the contributions as described in Lemma \ref{type_3_contribution}. Note, in this case, that the strata $\barH_{(\Gamma,\Gamma')}$ have the expected dimension.

Up to a constant factor (depending on $g$ and $d$ but not $(d_1,d'_1)$), the relevant contribution to $\xi^{*}\barH_{g/1,d,(m_2)^2}$ may be expressed as the pushforward of the fundamental class by the source map
\begin{equation*}
\phi:\coprod_{d_1+d'_1=11}\barH_{(1,1)/1,(d_1,d'_1),11}^{\circ}\to \barM_{1,11}\times \barM_{1,11},
\end{equation*}
where $\barH_{(1,1)/1,(d_1,d'_1),11}^{\circ}$ denotes the space of disconnected covers $X_1\coprod X'_1\to Y_1$, consisting of isogenies of degrees $d_1,d'_1$ and 11 pairs of points on $X_1,X'_1$ with equal image. Note, as in Lemma \ref{hecke_lemma}, that we do not label here the other $d-2$ points in each of these 11 special fibers.

We have a Cartesian diagram
\begin{equation*}
\xymatrix{
\barH_{(1,1)/1,(d_1,d'_1),11}^{\circ} \ar[r] \ar[d]^{\delta} & \barH_{1/1,d_1,11}^{\circ} \times \barH_{1/1,d'_1,11}^{\circ} \ar[d]^{(\delta,\delta)} \\
\barM_{1,11}\ar[r]^(0.4){\Delta}& \barM_{1,11}\times \barM_{1,11}
}
\end{equation*}
That is, $\barH_{(1,1)/1,(d_1,d'_1),11}^{\circ}$ parametrizes pairs of 11-pointed isogenies, with an isomorphism between the targets. In particular, we have
\begin{equation*}
\phi_{*}([\barH_{(1,1)/1,(d_1,d'_1),11}^{\circ}])=(\phi,\phi)_{*}(\delta,\delta)^{*}([\Delta]),
\end{equation*}
where the maps on the right hand side come from the correspondence
\begin{equation*}
\xymatrix{
\barH_{1/1,d_1,11}^{\circ} \times \barH_{1/1,d'_1,11}^{\circ} \ar[d]^{(\delta,\delta)} \ar[r]^(0.55){(\phi,\phi)} & \barM_{1,11}\times \barM_{1,11}  \\
\barM_{1,11}\times \barM_{1,11} &
}
\end{equation*}
arising as the product of correspondences from Lemma \ref{hecke_lemma}.

Finally, consider the K\"{u}nneth decomposition of the diagonal class $[\Delta]$. The terms consisting of pairs of even-dimensional classes have TKD both before and after applying the correspondence by Lemma \ref{even_algebraic}. By Lemma \ref{odd_cohomology}, the remaining terms are, up to a non-zero constant multiple,
\begin{equation*}
-\omega\otimes\overline{\omega}-\overline{\omega}\otimes\omega.
\end{equation*}

By Lemma \ref{hecke_lemma}, the correspondence acts by $\tau(d_1)\tau(d'_1)$ on this piece, and summing over all pairs $(d_1,d'_1)$, we find that the resulting class has non-zero odd contributions whenever the $d$-th coefficient $a_d$ of $\eta(q)^{48}$ is non-zero. In particular, it fails to have TKD, completing the proof.
\end{proof}

\begin{rem}
The modularity of the non-tautological contribution of the intersection of the $d$-elliptic cycle $\barH_{g/h,d,m_2}$ with the $\xi$ is consistent with the main conjecture of \cite{lian_qmod}, which predicts that the classes $\barH_{g/h,d,m_2}$ themselves are quasi-modular in $d$.
\end{rem}

\begin{cor}\label{elliptic_open}
Suppose that $d\ge2$ , $g\ge2$, $g+m_2=12$, and $a_d\neq0$. Then, $\cH_{g/1,d,(m_2)^2}\in H^{22}(\cM_{g,2m_2})$ is non-tautological.
\end{cor}

\begin{proof}
The proof is identical of \cite[Theorem 2]{vanzelm}: pullbacks of boundary cycles of (complex) codimension 11 have TKD on $\barM_{1,11}\times\barM_{1,11}$, so the failure of $\barH_{g/1,d,(m_2)^2}$ to have TKD upon this pullback persists after adding any combination of boundary cycles.
\end{proof}

\subsection{Induction on genus}\label{g_induction}

\begin{thm}\label{elliptic_general}
Suppose that $d\ge2$, $g\ge2$, $g+m_2\ge12$, and furthermore that $a_d\neq0$. Then, $\barH_{g/1,d,(m_2)^2}\in H^{*}(\barM_{g,2m_2})$ is non-tautological.
\end{thm}

We prove Theorem \ref{elliptic_general} by induction on $g$. When $2\le g\le 12-m_2$, the result follows by Proposition \ref{elliptic_equal12} and Lemma \ref{add_pair}.

Now, suppose $g>12$, so that in particular $(g-1)+m_2\ge12$. We consider the pullback of $\barH_{g/1,d,(m_2)^2}$ to the boundary divisor $\xi:\barM_{g-1,2m_2+1}\times\barM_{1,1}\to\barM_{g,2m_2}$. More precisely, let $b=2g-2+m_2$ be the dimension of $\barH_{g/1,d,m_2}$, also equal to the number of marked points on the target curve. Then, we consider the projection $\xi^{*}(\barH_{g/1,d,(m_2)^2})_{b-2,1}$ of $\xi^{*}(\barH_{g/1,d,(m_2)^2})$ to the factor
\begin{equation*}
H_{2(b-2)}(\barM_{g-1,2m_2+1})\otimes H_{2}(\barM_{1,1})\subset H_{2(b-1)}(\barM_{g-1,2m_2+1}\times\barM_{1,1}),
\end{equation*}
of the K\"{u}nneth decomposition. The factor $H_{2}(\barM_{1,1})$ is spanned by the fundamental class; we show by induction that the resulting class on $H_{2(b-2)}(\barM_{g-1,2m_2+1})$ is non-tautological, so that $\xi^{*}(\barH_{g/1,d,(m_2)^2})$  fails to have TKD.

Consider the usual diagram
\begin{equation*}
\xymatrix{
\coprod \barH_{(\Gamma,\Gamma')} \ar[r] \ar[d] & \barH_{g/1,d,(m_2)^2} \ar[d]^{\phi}\\
\coprod \barM_{A} \ar[r] \ar[d] & \barM_{g,N} \ar[d]^{\pi}\\
\barM_{g-1,2m_2+1}\times\barM_{1,1} \ar[r]^(0.63){\xi} & \barM_{g,2m_2}
}
\end{equation*}

Let $[f:X\to Y]$ be a general point of $\barH_{(\Gamma,\Gamma')}$. Because the graphs $A$ have only one edge, by the genericity condition from Proposition \ref{hm_cartesian_sets}, $Y$ may only have one node, which must be separating by Lemma \ref{sep_node}. Thus, $Y$ is the union of a smooth genus 1 component $Y_1$ and a smooth rational component $Y_0$. In addition, $\barH_{(\Gamma,\Gamma')}$ is pure of codimension 1 in $\barH_{g/1,d,(m_2)^2}$, so in particular the intersection in the upper square occurs in the expected dimension.

Let $X_1,X_{g-1}$ be the subcurves of $X$ of genus $1,g-1$, respectively, corresponding to the pieces parametrized by the factors of $\barM_{g-1,2m_2+1}\times\barM_{1,1}$.

We consider two cases:
\begin{enumerate}
\item At least one component of $X_1$ maps to $Y_1$, and
\item $X_1$ maps entirely to $Y_0$.
\end{enumerate}

\begin{lem}\label{d-ell_higher_genus_type1}
The contributions to $\xi^{*}(\barH_{g/1,d,(m_2)^2})_{b-2,1}$ from strata $\barH_{(\Gamma,\Gamma')}$ whose general point satisfies (i) have TKD.
\end{lem}

\begin{proof}
As its genus is 1, the subcurve $X_1$ can contain only one component over $Y_1$, an elliptic component mapping via an isogeny of degree $d'\le d$. One of the pre-images of the nodes of $Y$ is chosen as the separating node parametrizing $\barM_{g-1,2m_2+1}\times\barM_{1,1}$, and at the others, we must attach rational tails, in order for the genus of $X_1$ to be equal to 1. The curve $X_{g-1}$ then has degree $d-d'$ over $Y_1$ and $d-d'+1$ over $Y_0$.

Recall that we are interested in the contribution
\begin{equation*}
\xi^{*}(\barH_{g/1,d,(m_2)^2})_{b-2,1}\in H_{2(b-2)}(\barM_{g-1,2m_2+1})\otimes H_{2}(\barM_{1,1})\cong H_{2(b-2)}(\barM_{g-1,2m_2+1})
\end{equation*}
The resulting class in $H_{2(b-2)}(\barM_{g-1,2m_2+1})$ may be computed by intersecting $\xi^{*}(\barH_{g/1,d,(m_2)^2})_{b-2,1}$ with $[\barM_{g-1,2m_2+1}]\times[\Spec(\bC)]$, which amounts in this case to imposing the condition that the elliptic component $X_1$ have fixed $j$-invariant.

This, in turn, gives a discrete set of choices for the isomorphism class of the target component $Y_1$. For each possible $Y_1$, and each possible generic topological type of a $X\to Y$, we get a contribution to 
\begin{equation*}
\xi^{*}(\barH_{g/1,d,(m_2)^2})_{b-2,1}\in H_{2(b-2)}(\barM_{g-1,2m_2+1})
\end{equation*}
given by the product of a Hurwitz locus for the \textit{fixed} targets $Y_1$ and and a Hurwitz locus for the rational target $Y_0$, pushed forward by a boundary morphism. In particular, by Lemmas \ref{taut_rational_target} and \ref{taut_fixed_target}, all such contributions are tautological.
\end{proof}

\begin{lem}\label{d-ell_higher_genus_type2}
All strata $\barH_{(\Gamma,\Gamma')}$ whose general point satisfies (ii) and which give non-zero contributions to $\xi^{*}(\barH_{g/1,d,(m_2)^2})_{b-2,1}$ have general point $[f:X\to Y]$ of the following form, also depicted in Figure \ref{Fig:g_induction}.

$X_{g-1}$ consists of a smooth genus $g-1$ component mapping to $Y_1$ with degree $d$, along with $d-2$ rational tails; at a ramification point, a smooth genus 1 curve $X_1$ is attached, and $X_1$ maps to $Y_0$ with degree 2.
\end{lem}

\begin{figure}[!htb]
     \includegraphics[width=.45\linewidth]{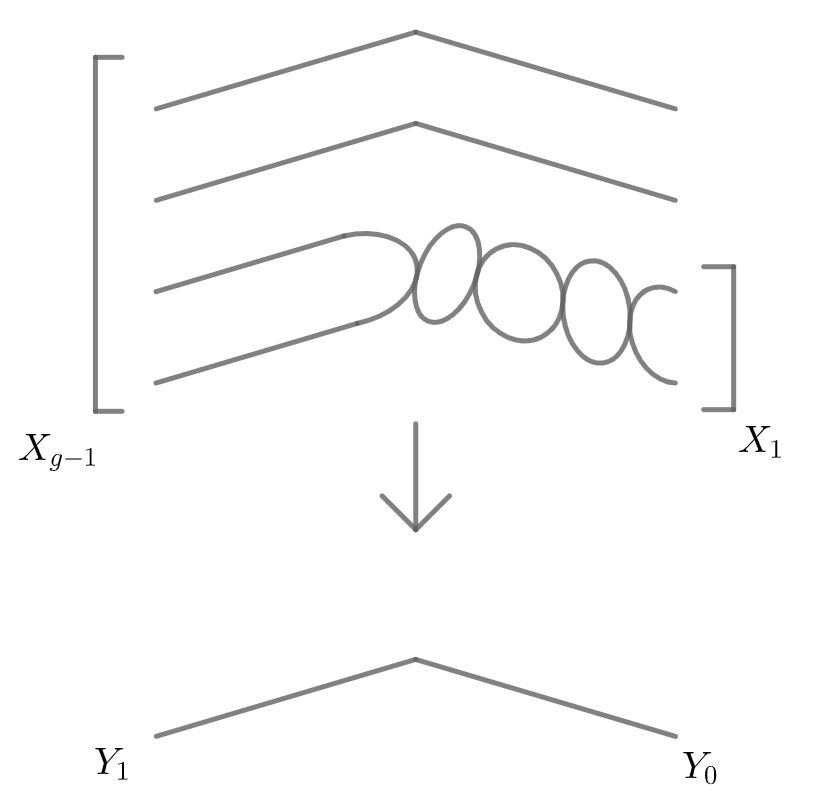}
          \caption{The only possible contribution of type (ii) to $\xi^{*}(\barH_{g/1,d,(m_2)^2})_{b-2,1}$. All marked points on the source lie on $X_{g-1}$.}\label{Fig:g_induction}
\end{figure}

\begin{proof}
If $X_1$ maps entirely to $Y_0$, then $X_1$ must be a smooth genus 1 curve. In order for the node at which $X_1,X_{g-1}$ meet to be separating, we need $X_1$ to be totally ramified over $Y_0$. 

All $2m_2$ of the (unforgotten) unramified marked points of $X$ are constrained to lie on $X_{g-1}$, so these points, as well as the $2g-2$ ramification points, can be associated in a well-defined way to one of $X_{1}$ and $X_{g-1}$. Let $b_1,b_{g-1}$, respectively, be the number of marked points appearing on these components, so that $b_1+b_{g-1}=b$. We have $b_1\ge3$, so $b_{g-1}\le b-3$.

On the other hand, let $b_{g-1,0},b_{g-1,1}$ be the number of marked points on $X_{g-1}$ mapping to $Y_0,Y_1$, respectively. Suppose that $b_{g-1,0}>0$. Then, a dimension count shows that the dimension of the image of $\barH_{(\Gamma,\Gamma')}$ upon projection to $\barM_{g-1,2m_{2}+1}$ is less than $b-2$. In particular, the contribution to $\xi^{*}(\barH_{g/1,d,(m_2)^2})_{b-2,1}$ is zero.

Thus, we find $b_{g-1,0}=0$, $b_{g-1,1}=d-3$, and $b_1=3$, from which we may conclude immediately.
\end{proof}

\begin{proof}[Proof of Theorem \ref{elliptic_general}]
By the previous two lemmas, all contributions to $\xi^{*}(\barH_{g/1,d,(m_2)^2})_{b-2,1}$ have TKD except possibly those coming from strata $\barH_{(\Gamma,\Gamma')}$ as described in Lemma \ref{d-ell_higher_genus_type2}, for which we get a positive multiple of $\barH_{(g-1)/1,(m_2)^2,1}$. By Lemma \ref{add_ram} and the inductive hypothesis, this class is non-tautological on $\barM_{g-1,2m_2+1}$, so $\xi^{*}(\barH_{g/1,d,(m_2)^2})_{b-2,1}$ fails to have TKD. In particular, $\barH_{g/1,d,(m_2)^2}$ is non-tautological.
\end{proof}

\section{Higher genus targets}\label{h_induction}

In this section, we complete the proof of Theorem \ref{main_thm}, by induction on $h$, with the base case given by Theorem \ref{elliptic_general}. As $d$ is fixed throughout, we will eventually require the same non-vanishing condition $a_d\neq0$.

\begin{prop}\label{higher_genus_target_main}
Suppose that $h\ge2$, $d\ge2$, $g\ge d$, $m_2\ge0$, $s\ge\max\{2,d-1\}$, and $m_d\ge s-1$. Suppose further that $\barH_{(g-d)/(h-1),d,(m_2)^2(m_d-s+2)^d}\in H^{*}(\barM_{g-d,2m_2+d(m_d-s+2)})$ is non-tautological (and in particular, that the cohomology group in question is non-zero and the Hurwitz locus is non-empty).

Then, $\barH_{g/h,d,(m_2)^2(m_d)^d}\in H^{*}(\barM_{g,2m_2+dm_d})$ is non-tautological.
\end{prop}

Consider the codimension $d$ stratum
\begin{equation*}
\xi:\barM_{g-d,2m_2+d(m_d-s+2)}\times(\barM_{1,s})^{d}\to\barM_{g,2m_2+dm_d}
\end{equation*}
parametrizing ``comb'' curves, that is, curves formed by attaching $d$ elliptic tails to a ``spine'' of genus $g-d$. We require that  $s-1$ of the $d$-tuples of unramified points lie on the elliptic tails, with one point of each $d$-tuple distributed to each tail. The remaining $d$-tuples are constrained to lie on the spine, as are all $m_2$ pairs of unramified points.

We will prove Proposition \ref{higher_genus_target_main} by showing that $\xi^{*}\barH_{g/h,d,(m_2)^2(m_d)^d}$ fails to have TKD.

Let $b=(2g-2)-d(2h-2)+(m_2+m_d)$ be the number of marked points on the target of a cover parametrized by $\barH_{g/h,d,(m_2)^2(m_d)^d}$, and let $B=(3h-3)+b$ be the dimension of $\barH_{g/h,d,(m_2)^2(m_d)^d}$. We will consider the projection $\xi^{*}(\barH_{g/h,d,(m_2)^2(m_d)^d})_{B-s-1,s-d+1}$ of $\xi^{*}(\barH_{g/h,d,(m_2)^2(m_d)^d})$ to
\begin{equation*}
H_{2(B-s-1)}(\barM_{g-d,2m_2+d(m_d-s+1)})\otimes H_{2(s-d+1)}((\barM_{1,s})^{d}) \subset H_{2(B-d)}(\barM_{g-d,2m_2+d(m_d-s+1)}\times(\barM_{1,s})^{d})
\end{equation*}
Note, in particular, that the condition $s\ge d-1$ ensures that $H_{2(s-d+1)}((\barM_{1,s})^{d})$ is non-trivial.

As usual, consider the diagram

\begin{equation*}
\xymatrix{
\coprod \barH_{(\Gamma,\Gamma')} \ar[r] \ar[d] & \barH_{g/h,d,(m_2)^2(m_d)^d} \ar[d]^{\phi}\\
\coprod \barM_{A} \ar[r] \ar[d] & \barM_{g,r} \ar[d]^{\pi}\\
\barM_{g-d,2m_2+d(m_d-s+1)}\times(\barM_{1,s})^{d} \ar[r]^(0.63){\xi} & \barM_{g,2m_2+dm_d}
}
\end{equation*}

\begin{lem}\label{higher_genus_target_contributions}
The $\barH_{(\Gamma,\Gamma')}$ which give non-zero contributions to $\xi^{*}(\barH_{g/h,d,(m_2)^2(m_d)^d})_{B-s-1,s-d+1}$ have general point $[f:X\to Y]$ of the following form, also depicted in Figure \ref{Fig:h_induction}.

$Y$ consists of two smooth components $Y_1,Y_{h-1}$ of genus $1,h-1$, respectively. Over $Y_{h-1}$, $X$ contains a single smooth connected component $X_{g-d}$ of genus $g-d$, and over $Y_1$, $X$ contains $d$ elliptic components mapping isomorphically to $Y_1$ (and attached at unramified points to $X_{g-d}$).
\end{lem}

\begin{figure}[!htb]
     \includegraphics[width=.45\linewidth]{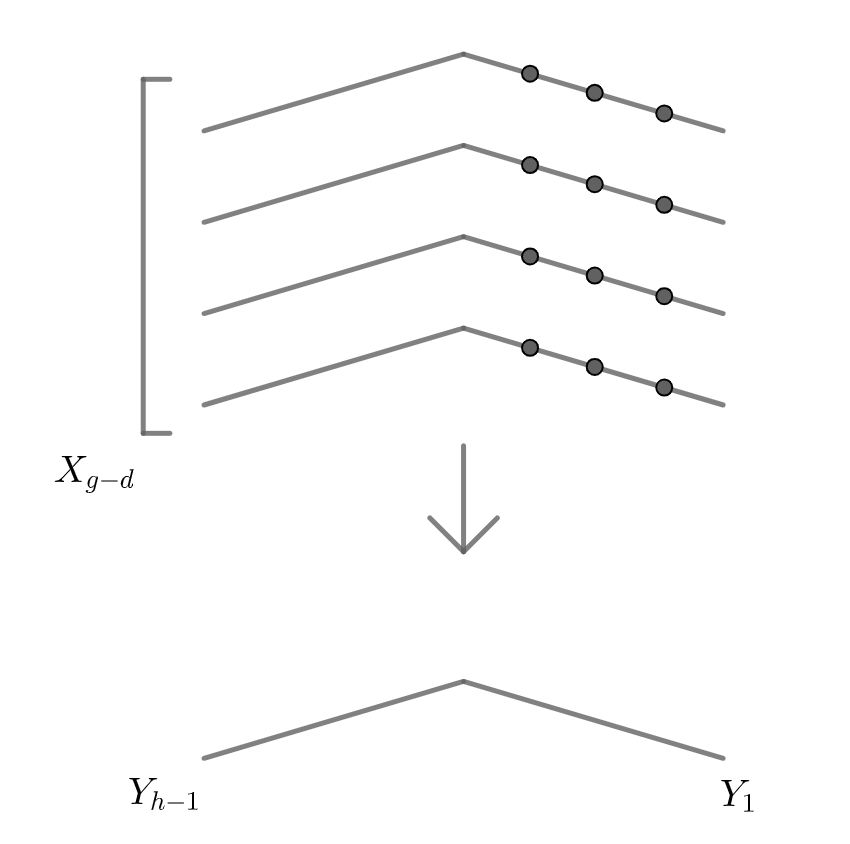}
          \caption{The only possible contribution to $\xi^{*}(\barH_{g/h,d,(m_2)^2(m_d)^d})_{B-s-1,s-d+1}$. Here, $s=4$, and all other marked points on the source lie on $X_{g-h}$.}\label{Fig:h_induction}
\end{figure}

\begin{proof}
Suppose $f:X\to Y$ is a general cover in a stratum $\barH_{(\Gamma,\Gamma')}$. Consider a marked point $y\in Y$ with marked $d$-tuple $x_1,\ldots,x_d$ of pre-images lying on the elliptic tails of $X$. Because $s\ge2$, at least one such marked fiber exists.

The points $x_1,\ldots,x_d$ must all lie on different components $X_1,\ldots,X_d$ of $X$, which must therefore map isomorphically to a component $Y_1\subset Y$. Note that all of these components must be tails, or else the valence of one of the elliptic vertices of $\Gamma$ would be greater than 1. Above the node of $Y_1$ corresponding to a half-edge of $\Gamma'$, at least one node must be chosen to correspond to a half-edge of $A$. In particular, $g(Y_1)=g(X_1)=\cdots=g(X_d)=1$. Furthermore, all $s-1$ of the marked points of $Y$ corresponding to marked $d$-tuples on elliptic tails must lie on $Y_1$, and the resulting $s$-marked elliptic curves $X_1,\ldots,X_d$ are all isomorphic.

Let $Y_{h-1}$ be the closure of $Y-Y_1$, over which the spine of $X$ lives. If the contribution of $\barH_{(\Gamma,\Gamma')}$ to $\xi^{*}(\barH_{g/h,d,(m_2)^2(m_d)^d})_{B-s-1,s-d+1}$ is non-zero, then the image of $\barH_{(\Gamma,\Gamma')}$ in $\barM_{g-d,2m_2+d(m_d-s+1)}$ must have dimension at least $B-s-1$. A parameter counting argument as we have carried out in the proofs of Lemmas \ref{add_ram}, \ref{type_3_contribution}, and \ref{d-ell_higher_genus_type2} shows that the $b-(s-1)$-pointed curve $Y_{g-1}$ must be smooth of genus $g-1$. In particular, the pre-image must be a smooth and connected curve of genus $g-d$, completing the proof.
\end{proof}

Lemma \ref{higher_genus_target_contributions} shows that the only non-zero contributions to $\xi^{*}(\barH_{g/h,d,(m_2)^2(m_d)^d})_{B-s-1,s-d+1}$ come from the diagram
\begin{equation*}
\xymatrix{
\barH_{(g-d)/(h-1),d,m_2+m_d-s+2}\times\barM_{1,s} \ar[r] \ar[d]^{(\phi,\Delta)} & \barH_{g/h,d,m_2+m_d} \ar[d]^{\phi}\\
\barM_{g-d,N-(s-2)d}\times(\barM_{1,s})^{d} \ar[r] \ar[d] & \barM_{g,N} \ar[d]\\
\barM_{g-d,2m_2+d(m_d-s+2)}\times(\barM_{1,s})^{d} \ar[r]^(0.63){\xi} & \barM_{g,2m_2+dm_d}
}
\end{equation*}

\begin{proof}[Proof of Proposition \ref{higher_genus_target_main}]
We apply the excess intersection formula in the top square; note that in the functorial fiber product, $\barH_{(g-d)/(h-1),d,m_2+m_d-s+2}\times\barM_{1,s} $ appears without non-reducedness, as the generic covers appearing in Lemma \ref{higher_genus_target_contributions} are unramified at the nodes. Recall from \S\ref{hm_boundary_int} that we need to pass to the normalization of $\barH_{(g-d)/(h-1),d,m_2+m_d-s+2}$. The dimensions of $\barH_{(g-d)/(h-1),d,m_2+m_d-s+2}$ and $\barM_{1,s}$ are $B-s-1,s$, respectively, and we are looking for the contribution in homological dimension $(B-s-1,s-d+1)$ on $\barM_{g-d,2m_2+d(m_d-s+1)}\times(\barM_{1,s})^{d}$. On the other hand, the intersection in the top square occurs in dimension $d-1$ greater than the expected.

Therefore, after applying the excess intersection formula, the piece of resulting the class on $\barM_{g-d,2m_2+d(m_d-s+1)}\times(\barM_{1,s})^{d}$ appearing in the desired pair of dimensions is a non-zero multiple of the pushforward of
\begin{equation*}
\barH_{(g-d)/(h-1),d,(m_2)^2(m_d-s)^d}\times \psi^{s-d+1},
\end{equation*}
where the $\psi$ class on $\barM_{1,s}$ is taken at the marked point to which the spine of $X$ is attached.

Therefore, $\xi^{*}(\barH_{g/h,d,(m_2)^2(m_d)^d})$ fails to have TKD, provided that $\psi^{s-d+1}\neq0$. However, note that, by the string equation,
\begin{equation*}
\int_{\barM_{1,s}}\psi_i^{s}=\int_{\barM_{1,1}}\psi\neq0,
\end{equation*}
where we have pushed forward by the map forgetting all but the $i$-th marked point. In particular, all smaller powers of $\psi$ are also non-zero.
\end{proof}

\begin{proof}[Proof of Theorem \ref{main_thm}]
The first claim, for $d$-elliptic loci, is Theorem \ref{elliptic_general}.

Now, suppose $h>1$ and $d=2$. Note in this case that $m_2=m_d\ge1$ by assumption. We prove the desired claim by induction on $h$ by applying Proposition \ref{higher_genus_target_main} with $s=2$. For $h=1$, we already have the same bounds for $h=1$, though there the condition $m_2\ge1$ is superfluous. Now, we have $g\ge2h$ and $g+m_2\ge 2h+10$, so $(g-2)\ge 2(h-1)$ and $(g-2)+m_2\ge 2(h-1)+10$, and we may apply the inductive hypothesis.

Finally, suppose $h>1$ and $d>2$; again, when $h=1$, we have stronger bounds after applying Lemma \ref{d_to_2}, so we may use this as the base case for induction on $h$, applying Proposition \ref{higher_genus_target_main} with $s=d-1$. The conditions $g\ge d$ and $m_d\ge s-1=d-2$ are easily checked to be satisfied given the hypothesis of the theorem, and the needed inequalities are still satisfied when $(g,h,d,m_2,m_d)$ are replaced by $(g-d,h-1,d,m_2,m_d-d+3)$, so the proof is complete.
\end{proof}

%
%
%
%

\end{document}